\documentclass{amsart}
\usepackage{cite}

\usepackage[english]{babel}
\usepackage{amssymb}
\usepackage{latexsym}
\usepackage{xypic}
\usepackage{multirow}

\newtheorem{theorem}{Theorem}[section]
\newtheorem{lemma}[theorem]{Lemma}
\newtheorem{proposition}[theorem]{Proposition}

\newtheorem{corollary}[theorem]{Corollary}

\theoremstyle{definition}
\newtheorem{definition}[theorem]{Definition}

\theoremstyle{remark}
\newtheorem{remark}[theorem]{Remark}

\numberwithin{equation}{section}

\begin{document}

\title{\sc{The limit set for discrete complex hyperbolic  groups }  }

\author{Angel Cano}
\address{ IMATE UNAM, Unidad Cuernavaca, Av. Universidad s/n. Col. Lomas de Chamilpa,
C.P. 62210, Cuernavaca, Morelos, M\'exico.}
\email{angelcano@im.unam.mx}

\author{Bingyuan Liu}
\address{University of California, Riverside, 900 University Ave, Riverside, CA 92521, USA.}
\email{bingyuan@ucr.edu }

\author{Marlon M. L\'opez}
\address{UERJ,  R. S\~{a}o Francisco Xavier, 524 - Maracan\~{a}, Rio de Janeiro, RJ, Brasil.}
\email{marlon.flores@uerj.br}

  \thanks{Supported by grants of the PAPPIT's projects IN106814,  IN101816 and Conacyt's project 164447} 




\begin{abstract}
Given a discrete subgroup $\Gamma$ of $PU(1,n)$ that acts by isometries on the  unit complex ball  $\Bbb{H}^n_{\Bbb{C}}$. In this 
setting a lot of work has been done in order to understand the  action of the group. However, when we look at the  action of $\Gamma$ 
on all of $ \Bbb{P}^n_{\Bbb{C}}$ little or nothing is known. In this paper, we study the action in the whole projective space  and 
we are able to  show  that its equicontinuity agrees with its Kulkarni discontinuity  set.  Moreover, in the non-elementary case, this set
turns out to be the largest open set on which the group acts properly and discontinuously. It can be described as the complement of the
union of all complex projective hyperplanes in $ \Bbb{P}^n_{\Bbb{C}}$  which are tangent to $\partial \Bbb{H}^n_{\Bbb{C}}$ at points 
in the Chen-Greenberg limit set $\Lambda_{CG}(\Gamma)$. 
\end{abstract}

\maketitle

\section*{Introduction}

The theory of Complex Kleinian groups  is still in its early developments. One unsolved problem is about the existence of largest open sets where a given group acts properly discontinuously. Little is known about this problem, see for instance 
\cite{CNS, CS,frances, mendez, SV1, SV2}. In this article, we answer this question in a  special case. We do it for complex hyperbolic groups. In this  case we prove:

\begin{theorem} \label{t:principal}
Let $\Gamma\subset PU(1,n)$ be  discrete subgroup, then the Kulkarni limit set  of $\Gamma$ can be described as the  hyperplanes tangent to $\partial \Bbb{H}^n_\Bbb{C}$ at  points in the Chen-Greenberg limit set of $\Gamma$, {\it i. e.}
\[
\Lambda_{Kul}(\Gamma)=\bigcup_{p\in \Lambda_{CG}(\Gamma)}p^\bot.
\]
Moreover, if $\Gamma$ is non-elementary, then $\Omega_{Kul}(\Gamma)$ agrees with the equicontinuity set of $\Gamma$ and is the largest open set on which the group acts properly and discontinuously. 
\end{theorem}

This result was proven (essentially) by J. P. Navarrete in \cite{jp} for n=2. In \cite{CS}, J. Seade and one of the authors studied the higher dimensional case. They proved that, for $\Gamma$  as above, the region of equicontinuity is as stated in the theorem, and they asked  whether the full statement of the theorem hold. In this article we answer the question affirmatively.\\

In the 2-dimensional scenario, the proof of Theorem \ref{t:principal} is a key step  to show that in the ``generic case" there is a well defined  notion of limit set, see \cite{CNS},  and we expect that  in the higher dimensional setting we get a similar result.  In a series of  forthcoming articles, we will clarify this assertion, see \cite{ACCM, ucan}.\\

The paper is organized as follows: in Section \ref{s:recall}, we review some general facts and introduce the notation used along the 
text. In Section   \ref{s:example}, we provide an example which depicts the problem of finding a maximum region where a group acts
properly and discontinuously.  In Section   \ref{s:lambda}, we provide a  lemma   that helps us to describe the dynamic of compact 
sets. In Section \ref{s:control}, we construct a group of transformations, called the control group,  which helps us to describe the 
induced dynamic in the Grassmanian $Gr(1,n)$. In Section \ref{s:main}, we provide  a  proof of the main theorem. Finally, in Section 
\ref{s:cartan}, we provide a relation between the control group and the Cartan angular invariant of triplets contained in the 
Chen-Greenberg limit set.\\
\section{Preliminaries}  \label{s:recall}

\subsection{Projective Geometry}
The complex projective space $\mathbb {P}^n_{\mathbb {C}}$
is defined as:
$$ \mathbb {P}^{n}_{\mathbb {C}}=(\mathbb {C}^{n+1}\setminus \{0\})/\Bbb{C}^* \,,$$
where $\Bbb{C}^*$ acts by  the usual scalar multiplication.
  This is   a  compact connected  complex $n$-dimensional
manifold  equipped with the Fubini-Study  metric $d_n$.

If $[\mbox{ }]:\mathbb {C}^{n+1}\setminus\{0\}\rightarrow
\mathbb {P}^{n}_{\mathbb {C}}$ is the quotient map, then a
non-empty set  $H\subset \mathbb {P}^n_{\mathbb {C}}$ is said to
be a projective subspace of dimension $k$  if there is a  $\mathbb {C}$-linear
subspace  $\widetilde H$ of dimension $k+1$ such that $[\widetilde
H\setminus \{0\}]=H$. The  projective subspaces of dimension
$(n-1)$ are called hyperplanes and the  complex projective subspaces of dimension 1
 are called   lines.  In this article,  $e_1,\ldots, e_{n+1}$ will denote the standard basis for $\Bbb{C}^{n+1}$.\\

Given a set of points $P$   in $\mathbb{P}^{n}_{\mathbb{C}}$, we
define:
$$Span( P) =\bigcap\{l\subset \mathbb{P}^n_{\mathbb{C}}\mid l \textrm{ is a projective subspace containing } P \}.$$ Clearly,  
$Span( P) $ is a projective subspace of $\mathbb{P}^{n}_{\mathbb{C}}$.  If $p,q$ are distinct points then  $Span(\{p,q\})$ is the 
unique complex line passing through them. In this case, we will write $\overleftrightarrow{p,q}$ instead of $Span(\{p,q\})$.

\subsection{ Projective and Pseudo-projective Transformations }
 It is clear  that every linear isomorphism of $\Bbb{C}^{n+1}$ defines a
holomorphic automorphism of $\mathbb{P}^{n}_{\mathbb{C}}$. Also, it is well-known that every holomorphic automorphism of $\mathbb{P}^{n}_{\mathbb{C}}$
arises in this way. The group of projective automorphisms of $\mathbb{P}^{n}_{\mathbb{C}}$ is defined:
$$PSL(n+1, \mathbb {C}) \,:=\, GL({n+1}, \Bbb{C})/\Bbb{C}^*,$$
where $\Bbb{C}^* $ acts by the usual scalar multiplication. Then $PSL(n+1, \mathbb {C})$ is a Lie group whose
elements are called projective transformations.
We denote   by $[[\mbox{  }]]: GL(n+1, \mathbb
{C})\rightarrow PSL(n+1, \mathbb {C})$    the quotient map. Given     $ \gamma  \in PSL(n+1, \mathbb
{C})$,  we  say that  $\widetilde \gamma  \in GL(n+1, \mathbb {C})$ is a  lift of $ \gamma  $ if 
$[[\widetilde  \gamma  ]]= \gamma  $.\\

\subsection{Complex Hyperbolic Groups} 
In the rest of the paper, we will be interested in studying those subgroups of $PSL(n+1,\Bbb{C})$ preserving  the unitary
complex ball.  We start  by considering the following Hermitian matrix: 
\[
H=
\left (
\begin{array}{lll}
  &   & 1\\
  & I_{n-1} &\\
1&   &
\end{array}
\right ),
\]
where $I_n$ denotes the identity matrix of size $(n-1)\times (n-1)$. We will set
\[
U(1,n)=\{g\in GL(n+1,\Bbb{C}): gHg^*=H\}.
\]
and $\langle,\rangle:\Bbb{C}^{n+1}\rightarrow \Bbb{C}$ the hermitian form induced by $H$. Clearly, $\langle,\rangle$ has signature 
$(1,n)$, $U(1,n)$ is the the group preserving $\langle,\rangle$, see \cite{neretin}. And the corresponding projectivization $PU(1,n)$ 
preserves the unitary complex ball:
\[
\Bbb{H}^n_\Bbb{C}=\{[w]\in \Bbb{P}^n_{\Bbb{C}}\mid \langle w,w\rangle <0\}
\]
 Given a group $\Gamma\subset PU(1,n)$, we define the following notion of limit set due to Chen and Greenberg, see \cite{CG}.
 
\begin{definition}
Let $\Gamma\subset PU(1,n)$, then $\Lambda_{CG}(\Gamma)$ is to be defined  as as the set of accumulation
points in $\partial \Bbb{H}^n_\Bbb{C}$ of the orbit of any point in $\Bbb{H}^n_\Bbb{C}$.
\end{definition}

 As in the Fuchsian groups case, it is not hard to show that $\Lambda_{CG}(\Gamma)$ does not depend on the choice of $x$ and
 $\Lambda_{CG}(\Gamma)$ has either 1,2 or infinite points.  A group is said to be non-elementary if  $\Lambda_{CG}(\Gamma)$ has 
 infinite points.\\
 
 In the following, given a projective subspace $P\subset \Bbb{P}^n_\Bbb{C}$ we will define 
\[
P^\bot=[\{w\in \Bbb{C}^{n+1}\mid \langle w,v\rangle=0 \textrm{ for all } v\in [P]^{-1} \}\setminus\{0\}].
\]
Also, we will say that $P$ is a Lagrangian plane if there is  a $\Bbb{R}$-vectorial subspace $R\subset \Bbb{C}^{n+1}$ of dimension 3, 
such that $[R\setminus \{0\}]=P$ and $\langle w,v\rangle\in \Bbb{R}$ for each $v,w\in R$. A non elementary  group 
$\Gamma\subset PU(1,n)$ is said to be $\Bbb{C}$-Fuchsian (resp. $\Bbb{R}$-Fuchsian) if there is a complex line 
(resp.  a Lagrangian plane) $\ell$ invariant under $\Gamma$.\\

Let us define the Cartan angular invariant. Given $([x],[y],[z])\in \partial (\Bbb{H}^n_\Bbb{C})^3$  a triplet of  distinct points, 
the Cartan angular invariant of this triplet is defined  as
\[
\Bbb{A}([x],[y],[z])=arg(-\langle x,y\rangle \langle y,z\rangle\langle z,x\rangle).
\] 
It can be shown, see \cite{goldman}, that  $2\vert \Bbb{A}([x],[y],[z])\vert \leq \pi$ and 
\begin{enumerate}
\item we have $2\Bbb{A}([x],[y],[z])=\pm \pi$ if and only if the points $[x],[y],[z]$ lie in a complex line, 

\item we have $\Bbb{A}([x],[y],[z])=0$ if and only if the points $[x],[y],[z]$ lie in a Lagrangian plane.
\end{enumerate}
We also, have the following result:
\begin{theorem}
Given $([x_1],[x_2],[x_3]),([\tilde{ x}_1],[\tilde{x}_2],[\tilde{x}_3])\in \partial (\Bbb{H}^n_\Bbb{C})^3$  be  triplets of  distinct points we have $\Bbb{A}([\tilde{ x}_1],[\tilde{x}_2],[\tilde{x}_3])=\Bbb{A}([x_1],[x_2],[x_3])$ if and only if there is a transformation $\gamma\in PU(1,n)$ such that $\gamma([x_i])=[\tilde{x}_i]$.
\end{theorem}
The following result will be useful, see \cite{knapp, neretin}.

\begin{theorem}[Cartan's  Decomposition]
For every $\gamma\in PU(1,n)$ there are elements $k_1,k_2\in K:=PU(n+1)\cap PU(1,n)$ and  a unique   $\mu(\gamma)\in PU(1,n)$,  such that $\gamma=k_1\mu(\gamma)k_2$ and $\mu(\gamma)$ has a lift in $SL(n+1,\Bbb{C})$ given by
\[
\left (
\begin{array}{lllll}
e^{\lambda(\gamma)}\\
                 & 1\\
                 &          & \ddots\\
                 &          &            &1\\
                 &          &            &     & e^{-\lambda(\gamma)}
\end{array}
\right ),
\]
where $\lambda(\gamma) \geq 0$.
\end{theorem}

\subsection{ Pseudo-projective Transformation}
 The space of linear transformations from $\Bbb{C}^{n+1}$ to $\Bbb{C}^{n+1}$, denoted by  $M(n+1,\Bbb{C})$,   is a linear complex space  of dimension
 $(n+1)^2$, where $GL(n+1,\Bbb{C})$ is an open dense set in $ M(n+1,\Bbb{C})$. Then
$PSL(n+1,\Bbb{C})$ is an open dense set in  $QP(n+1,\Bbb{C})= (M(n+1,\Bbb{C})\setminus \{0 \})/\Bbb{C}^*$ called the space of pseudo-projective   maps. Let $\widetilde M:\mathbb {C}^{n+1}\rightarrow \mathbb {C}^{n+1}$ be a non-zero
 linear transformation. Let   $Ker(\widetilde M)$ be its kernel and  $Ker([[\widetilde M]])$
 denote the respective projectivization,
 then  $\widetilde M$ induces a well defined map 
  $[[\widetilde M]]:\mathbb {P}^{n}_\mathbb {C}\setminus Ker([[\widetilde M]]) \rightarrow
 \mathbb {P}^{n}_\mathbb {C}$   by:
 $$
[[\widetilde M]]([v])=[\widetilde M(v)]\,.
 $$
 The following  result provides a link between pointwise convergence in $QP(n+1,\Bbb{C})$  and uniform convergence in a projective space.

\begin{proposition} [See \cite{CS}]  \label{p:completes}
Let  $( \gamma_m)_{m\in \mathbb {N}}\subset PSL(n+1,\mathbb {C})$
be a sequence of distinct elements, then 
\begin{enumerate}
\item  There is a subsequence $( \tau_m)_{m\in \mathbb {N}}\subset ( \gamma_m)_{m\in \mathbb {N}}$ and  $\tau_0\in  M(n+1,\Bbb{C})\setminus \{0\}$ such that $\tau_m\xymatrix{ \ar[r]_{m \rightarrow
\infty}&}  \tau_0  $ as points in $QP(n+1,\Bbb{C})$.

\item If $(\tau_m)_{m\in \mathbb {N}}$ is the sequence  given by the previous part of this lemma, then 
$\tau_m  \xymatrix{ \ar[r]_{m \rightarrow
\infty}&}  \tau_0  $, as   functions,   uniformly on compact sets of 
 $\mathbb
{P}^n_\mathbb {C}\setminus Ker( \tau_0  )$.

\end{enumerate}
\end{proposition}

\subsection{The Grassmanians} 
Let $0\leq k< n$,  we define the Grassmanian   $Gr(k,n)$ as the space  of all  $k$-dimensional projective subspaces  of $\mathbb {P}^{n}_{\Bbb{C}}$ endowed with the Hausdorff topology. One has  that $Gr(k,n)$ is a  compact, connected complex manifold of dimension $k(n-k)$. 
A method to realize the Grassmannian  $Gr(k,n)$  as a subvariety of the projective space of the $(k+1)$-th exterior power of $ \Bbb{C}^{n+1}$, in symbols $P(\bigwedge^{k+1} \Bbb{C}^{n+1})$ is done by the so called Pl\"ucker embedding which is given by:
\[
\begin{array}{l}
\iota:Gr(k,n)\rightarrow P(\bigwedge^{k+1} \Bbb{C}^{n+1})\\
\iota(V)\mapsto  [v_1\wedge \cdots \wedge v_{k+1}]
\end{array}
\] 
where $Span(\{ v_1, \cdots,  v_{k+1}\})= V$. We can make $PSL(n+1,\Bbb{C})$ act on $Gr(k,n)$  and $P(\bigwedge^{k+1} \Bbb{C}^{n+1})$, which 
makes $\iota$  a  $PSL(n+1,\Bbb{C})$-equivalent  embedding.\\

\subsection{Kulkarni Limit Set}
When we look at the action of a group acting on a general topological space, in general there are no natural notions of limit set. Hence, we study the notion introduced by Kulkarni. 

\begin{definition}[see   \cite{kulkarni} ] \label{d:lim}
 Let $\Gamma\subset   PSL(n+1,\mathbb{C})$ be a subgroup. We  define
 
\begin{enumerate}
\item  The set $L_0(\Gamma)$ as the closure of the set of points in $\Bbb{P}^n_{\Bbb{C}}$ with infinite isotropy
group.

\item The set $L_1(\Gamma)$  as the closure of the set of cluster points of $\{g(z) : g \in \Gamma\}$
where $z$ runs over $\Bbb{P}^n_{\Bbb{C}}\setminus L_0(\Gamma)$.

\item The set 
$\Lambda(\Gamma)=L_0(\Gamma)\cup L_1(\Gamma)$.

\item  The set $L_2(\Gamma)$ as  the closure of cluster  points of $\Gamma
K$  where $K$ runs  over all  the compact sets in
$\mathbb{P}^n_{\mathbb{C}}\setminus \Lambda(\Gamma)$.

\item The  \textit{ Kulkarni's limit set } of $\Gamma$  as:  $$\Lambda_{Kul} (\Gamma) = \Lambda(\Gamma) \cup L_2(\Gamma).$$ 
\item The \textit{ Kulkarni's discontinuity region} of $\Gamma$  as:
$$\Omega_{Kul}(\Gamma) = \mathbb{P}^n_{\mathbb{C}}\setminus
\Lambda_{Kul}(\Gamma).$$
\end{enumerate}

\end{definition}

The Kulkarni's limit set has the following properties,  for a more  detailed  discussion on this topic in the 2 dimensional setting
see \cite{CNS}. 

\begin{proposition}[See \cite{CNS, CS, kulkarni}] \label{p:pkg}
Let    $\Gamma$  be a complex  Kleinian group. Then:

\begin{enumerate}

\item The sets\label{i:pk2}
$\Lambda_{Kul}(\Gamma),\,\Lambda(\Gamma),\,L_2(\Gamma)$
are  $\Gamma$-invariant closed sets. 

\item \label{i:pk3} The group $\Gamma$ acts properly 
  discontinuously on  $\Omega_{Kul}(\Gamma)$. 

\item \label{i:pk4} Let $\mathcal{C}\subset\mathbb{P}^n_{\mathbb{C}}$ be
a closed $\Gamma$-invariant set such that  for every compact set $K\subset
\mathbb{P}^n_{\mathbb{C}}\setminus \mathcal{C}$, the set of cluster points
of  $\Gamma K$ is contained in $\Lambda(\Gamma)\cap \mathcal{C}$, then
$\Lambda_{Kul}(\Gamma)\subset \mathcal{C}$.

\item The equicontinuity set of $\Gamma$ is contained in $\Omega_{Kul}(\Gamma)$.
\end{enumerate}
\end{proposition}

\section{A starting example}\label{s:example}
Let us consider the element $\gamma\in PU(1,n)$ induced by the  following matrix
\[
\widetilde\gamma
=
\left (
\begin{array}{lll}
2 &             &\\
   & I_{n-1} &\\
   &             &2^{-1}
\end{array}
\right )
\]
and $\Gamma=\langle\langle \gamma\rangle \rangle$.
It is not hard to show that 
$$Eq(\Gamma)=\Bbb{P}^n\setminus (\{e_1\}^\bot\cup \{e_{n+1}\}^\bot )=\Omega_{Kul}(\Gamma).$$
On the other hand,  when one tries to determine if the previous sets are  maximal  open sets on which  $\Gamma$ acts properly 
discontinuously, we find the following phenomena: given $U,W\subset Span(\{e_2,\ldots, e_n\})=\mathcal{L}$ disjoint open sets such 
that $\overline U\cup \overline W=\mathcal{L}$, define 
\[
\begin{array}{l}
\mathcal{U}=\{\overleftrightarrow{e_1,u}\mid u\in \overline U\}\\
\mathcal{V}=\{\overleftrightarrow{e_{n+1},v}\mid v\in \overline V\}.\\
\end{array}
\]
Is not hard to show $\Bbb{P}^n\setminus(\mathcal{U}\cup \mathcal{V})$ is a maximal open set on which $\Gamma$ acts properly discontinuously and every maximal open set for the action of $\Gamma$ arises in this way.

\section{The lambda lemma} \label{s:lambda}
In this section, we develop a tool that will enable us to determine the accumulation points  of the orbit of compact sets, 
compare  results with those in \cite{frances, mendez}.\\

\begin{definition}
Let $(\gamma_m)\subset PU(1,n)$ be a sequence of distinct elements, we will say that $(\gamma_m)$ tends simply  to the infinite if:
\begin{enumerate} 
\item The sequences of  compact factors in the Cartan decomposition of $(\gamma_m)$ converge in $U(n+1)\cap U(1,n)$.
\item The sequence $\lambda(\gamma_m)$ converges to infinity.
\end{enumerate}
\end{definition}

Given a discrete group $\Gamma\subset PU(1,n)$ and a sequence $(\gamma_m)\subset  \Gamma$ of distinct elements, there is a
subsequence $(\tau_m)\subset (\gamma_m)$ tending simply to infinity. \\

\begin{definition}
Let $(\gamma_m)\subset PU(1,n)$ be a sequence tending simply to infinity and $x\in \Bbb{P}^n_\Bbb{C}$, we define:
\[ 
\mathcal{D}_{(\gamma_m)}(x) =
\bigcup
\{
 \textrm{ accumulation points of } (\gamma_m(x_m))
\}.
\]
The union is taken over all sequences converging to $x$.
\end{definition}

Clearly, if $\Gamma\subset PU(1,n)$ is a discrete group, $\Omega\subset \Bbb{P}^n_\Bbb{C}$ is an open set on which $\Gamma$ acts properly discontinuously,  $x\in \Omega$ and $(\gamma_m)\subset \Gamma$ is any sequence tending simply to infinity,  then $\mathcal{D}_{(\gamma_m)}(x)\subset \Bbb{P}^n_\Bbb{C}$.\\
 
Before we provide the main result of this section, let us make a brief digression 
to some projective geometry

\begin{definition}
Given a hyperplane $\mathcal{H}\subset \Bbb{P}^n_{\Bbb{C}}$ and a point 
$p\in \mathcal{H}$ we define
\[
\mathcal {H}(p)=
\{
\ell\in Gr_1(\Bbb{P}^n_\Bbb{C}): p\in \ell \subset \mathcal{H}
\}.
\]
\end{definition}

The following proposition is straightforward.

\begin{lemma}
Given  $p\in \partial \Bbb{H}^n_\Bbb{C}$, define  the following function in $p^\bot(p)\times p^\bot(p)$:

\[
d_{p}(\ell_1,\ell_2)=
\arccos
\left (
\sqrt
{
\frac{\langle q_1,q_2\rangle\langle q_1,q_2\rangle}{\langle q_1,q_1\rangle\langle q_2,q_2\rangle}
}
\right )
\]
where $q_1,q_2\in \Bbb{C}^{n+1}\setminus \{0\}$ are points satisfying
$\ell_i=\overleftrightarrow{p,[q_i]}$. Then $(p^\bot(p),d_p)$ is a metric space isometric to $(\Bbb{P}^{n-1}_\Bbb{C},d_{n-1})$ .
\end{lemma}

\begin{definition}
Given $p,q\in\partial \Bbb{H}^n_\Bbb{C}$, we will denote by $Isom(p,q)$ the set  of isometries  from $(p^\bot(p),d_p)$ to the  space $(q^\bot(q),d_q)$ 
\end{definition}

The following  proposition will be crucial along the paper, see also \cite{CS, frances,  mendez, jp}.

\begin{proposition} \label{p:llemaleq}
Let $\Gamma\subset PU(1,n)$ be a discrete group and  $(\gamma_m)\subset \Gamma$   a sequence tending simply to infinity, then there are    two pseudo projective transformations $\tau,\vartheta$   such that

\begin{enumerate}
\item \label{i:llemaleq1} We have  $\gamma_m  \xymatrix{ \ar[r]_{m \rightarrow  \infty}&} \tau$ and $\gamma_m^{-1}  \xymatrix{ \ar[r]_{m \rightarrow  \infty}&} \vartheta$ .

\item \label{i:llemaleq2} The sets  $Im(\tau)$ and $Im(\vartheta)$ are points in the Chen-Greenberg limit set satisfying $Im(\tau)^\bot= Ker(\vartheta)$  and 
$Im(\vartheta)^\bot= Ker(\tau).$
\end{enumerate}
Moreover, there is a projective  equivalence $\phi:Im(\vartheta)^\bot(Im(\vartheta))\rightarrow Im(\tau)^\bot(Im(\tau))$ satisfying:
\begin{enumerate}
\item[(a)]  \label{i:llemaleq4}  The equivalence $\phi$ belongs to $Isom(Im(\vartheta), Im(\tau))$.

\item[(b)]  \label{i:llemaleq5} 
Given   $\ell\in Im(\vartheta)^\bot(Im(\vartheta)) $,  and   $y\in \ell \setminus Im(\vartheta)$,  we know 

$$
\mathcal{D}_{(\gamma_m)}(x)=\phi(\ell).
$$

\item[(c)]  \label{i:llemaleq6} 
Given   $\ell\in Im(\tau)^\bot(Im(\tau)) $,  and   $y\in \ell \setminus Im(\tau)$,  we know 

$$
\mathcal{D}_{(\gamma_m^{-1})}(x)=\phi^{-1}(\ell).
$$

 \end{enumerate}
\end{proposition}

\begin{proof} Let us show part  (\ref{i:llemaleq1}). By the Cartan's decomposition theorem  there are  sequences $(\alpha_m)\in \Bbb{R}^+$, $(\kappa_m),(\tilde \kappa_m)\in K=U(n+1)\cap U(1,n)$, such that 
\begin{equation} \label{e:dec}
\gamma=
\left [
\kappa_m
\left (
\begin{array}{lllll}
e^{\alpha_m}\\
                 & 1\\
                 &          & \ddots\\
                 &          &            &1\\
                 &          &            &     & e^{-\alpha_m}
\end{array}
\right )
\tilde\kappa_m
\right ].
\end{equation}
We can assume that $\kappa_m \xymatrix{ \ar[r]_{m \rightarrow  \infty}&}\kappa_1\in  K $, $\tilde\kappa_m\xymatrix{ \ar[r]_{m \rightarrow  \infty}&} \kappa_2\in K$ and $\alpha_m\xymatrix{ \ar[r]_{m \rightarrow  \infty}&} \infty$. Equation \ref{e:dec} shows that    

\begin{equation}\label{e:cd}
\begin{array}{l}
\gamma_m  \xymatrix{ \ar[r]_{m \rightarrow  \infty}&} \tau=
\left [
\kappa_1
\left (
\begin{array}{lllll}
1   \\
                 & 0\\
                 &          & \ddots\\
                 &          &            &0\\
\end{array}
\right )
\kappa_2
\right ]\\ \\
\gamma^{-1}_m \xymatrix{ \ar[r]_{m \rightarrow  \infty}&} \vartheta=
\left [
\kappa_2^{-1}
\left (
\begin{array}{lllll}
 0\\
                           & \ddots\\
                           &       & 0\\
                           &       &            &1\\
\end{array}
\right )
\kappa^{-1}_1
\right ],
\end{array}
\end{equation}
and it shows part (\ref{i:llemaleq1}).\\

Let us show part (\ref{i:llemaleq2}). Observe that equation \ref{e:cd}  yields
\[
\begin{array}{l}
Im(\tau)=\kappa_1(e_{1} )\\
Im(\vartheta)=\kappa_2^{-1}(e_{n+1} )\\
Ker(\tau)=\kappa_2^{-1}(Span (e_2,\ldots,e_{n+1} ))\\
Ker(\vartheta)=\kappa_1(Span (e_1,\ldots,e_{n} )),\\
\end{array}
\]
which shows part (\ref{i:llemaleq2}).  \\

 Now  let us show part (a). Let us define
\[
\begin{array}{l}
\phi:Im(\vartheta)^\bot(Im(\vartheta))\rightarrow Im(\tau)^\bot(Im(\tau))\\
 \phi(\ell)=\kappa_1(H(\kappa_2(\ell))).\\
\end{array}
\]
it is not hard to see that $\phi$ is an isometry from $(Im(\vartheta)^\bot(Im(\vartheta)),d_{Im(\vartheta)})$ to the  space $(Im(\tau)^\bot(Im(\tau)),d_{Im(\tau)})$. \\

 Now let us show part (b). Let $\ell\in Im(\vartheta)^\bot (Im(\vartheta))$, $w\in Im(\vartheta)^\bot \setminus Im(\vartheta)$, and $(x_m)\subset \Bbb{P}^n_\Bbb{C}$ such that 
$x_m \xymatrix{ \ar[r]_{m \rightarrow  \infty}&} w$, then $\kappa_2(x_m)\xymatrix{ \ar[r]_{m \rightarrow  \infty}&} \kappa_2(w)$, where $\kappa_2(w)\in Span(e_2,\ldots, e_{n+1} )\setminus \{e_{n+1}\}$,   $\kappa_2(x_m)=[x_{1m},\ldots,x_{n+1,m}]$,  $\kappa_2(w)=[0, x_2,\ldots,x_{n+1}]$,  
$\sum_{j=2}^{n}\vert x_j\vert\neq 0 $ and 
$ x_{jm}\xymatrix{ \ar[r]_{m \rightarrow  \infty}&} x_j. $ Then

\begin{equation}\label{e:ac}
\left [
\begin{array}{lllll}
e^{\alpha_m}\\
                 & 1\\
                 &          & \ddots\\
                 &          &            &1\\
                 &          &            &     & e^{-\alpha_m}
\end{array}
\right ]
\left
[
\begin{array}{c}
x_{1m}\\
\vdots\\
x_{n+1,m}
\end{array}
\right ]
=
\left
[
\begin{array}{c}
e^{\alpha_m}x_{1m}\\
x_{2m}\\
\vdots\\
x_{n,m}\\
e^{\alpha_m}x_{n+1,m}
\end{array}
\right ]
=w_m,
\end{equation}
thus the accumulation points of $(w_m)$ lie on the line $\overleftrightarrow{e_1,[0,x_2,\ldots,x_n,0]}$, thus the accumulation points of $(\gamma_m(x_m))$ lie on the line $$\overleftrightarrow{Im(\tau),\kappa_1([0,x_2,\ldots,x_n,0])}=\phi(\ell).$$
Now, the proof of this part follows from Equation \ref{e:ac}.\\

In order to conclude the proof, one must observe that the proof of part (c) is similar to the proof of part (a).
\end{proof}

The proofs of the next two corollaries follow straightforwardly.

\begin{corollary} [See \cite{jp}] Let $\Gamma\subset PU(2,1)$ be a discrete group non-elementary group, then
\begin{enumerate}

\item The Kulkarni discontinuity region $\Omega_{Kul}(\Gamma)$ is the largest open set on  which $\Gamma$ acts properly and discontinuously and agrees with $Eq(\Gamma)$.

\item The Kulkarni limit set  of $\Gamma$ can be described as the  hyperplanes tangent to $\partial \Bbb{H}^2_\Bbb{C}$ at  points in the Chen-Greenberg limit set of $\Gamma$, {\it i. e.}

\[
\Lambda_{Kul}(\Gamma)=\bigcup_{p\in \Lambda_{CG}(\Gamma)}p^\bot.
\]
\end{enumerate}
\end{corollary}

\begin{corollary} \label{c:eqpu}[See \cite{CS}] Let $\Gamma\subset PU(1,n)$ be a discrete group, then the complement of the equicontinuity region of $\Gamma$ are the  hyperplanes tangent to $\partial \Bbb{H}^n_\Bbb{C}$ at  points in the Chen-Greenberg limit set of $\Gamma$, {\it i. e.}

\[
\Bbb{P}^n_\Bbb{C}\setminus Eq(\Gamma)=\bigcup_{p\in \Lambda_{CG}(\Gamma)}p^\bot.
\]
\end{corollary}
\section{The control group}\label{s:control}
As we will see, in order to understand the dynamic of $\Gamma$ it will suffice to understand the dynamic in $Gr(1,n)$. Such thing  can 
be done by introducing the  action of an adequate group. Let us consider the following definition.
 
\begin{definition}
Let $\Gamma\subset PU(1,n)$ be a discrete non-elementary group and $x,y\in \Lambda_{CG}(\Gamma)$. A transformation $\gamma:y^\bot(y) \rightarrow  x^\bot(x) $ is called a Cartan map for $(y,x,\Gamma)$ if there is a sequence $(\gamma_m)\subset \Gamma$ tending simply to infinity  such that 
\begin{enumerate}
\item  We have $\gamma_m \xymatrix{ \ar[r]_{m \rightarrow  \infty}&} x$ uniformly on compact  sets of $\Bbb{P}^n_\Bbb{C}\setminus y^\bot $.
\item Also $\gamma_m^{-1} \xymatrix{ \ar[r]_{m \rightarrow  \infty}&} y$ uniformly on compact  sets of $\Bbb{P}^n_\Bbb{C}\setminus x^\bot $.
\item For each $\ell\in y^\bot(y)$ and each $w\in\ell$ we have:
\[
\mathcal{D}_{(\gamma_m)}(w)=\gamma(\ell)
\]
\item For each $\ell\in x^\bot(x)$ and each $w\in\ell$ we have:
\[
\mathcal{D}_{(\gamma_m^{-1})}(w)=\gamma^{-1}(\ell)
\]
\end{enumerate}
The  sequence $(\gamma_m)$ is called the Cartan sequence associated to $\gamma$.
\end{definition}

The following lemma will be useful.

\begin{lemma} [See \cite{kamiya}]\label{l:2tran}
Let $\Gamma\subset PU(1,n)$ be a discrete non-elementary group and points  $x_+,x_-\in \Lambda_{CG}(\Gamma)$, then there is a sequence  $(\gamma_m)\subset \Gamma$ of distinct elements  such that 
\[
\gamma_m ^{\pm 1}\xymatrix{ \ar[r]_{m \rightarrow  \infty}&} x_\pm
\]
 uniformly  on compact set of $\Bbb{H}^n_{\Bbb{C}}$. 
\end{lemma}

\begin{lemma}
Let $\Gamma\subset PU(1,n)$ be a discrete and non elementary group and $x,y\in \Lambda_{CG}(\Gamma)$ be distinct points, then
\[
\Gamma(y,x)=
\left \{
\gamma: y^\bot (y)\rightarrow x^\bot(x)\mid \gamma 
\textrm{ is a cartan  map for } (y,x,\Gamma)
\right \}
\]
is a non-empty  subset of $Isom(y,x)$.
\end{lemma}

\begin{proof}
Let us show that $\Gamma(y,x)$  is non-empty. Let  $(\gamma_m)\subset \Gamma$ be the sequence given by Lemma \ref{l:2tran}.  
By Lemma \ref{p:llemaleq},  we can assume that  $(\gamma_m)$ tends simply to infinity, $\gamma_m \xymatrix{ \ar[r]_{m \rightarrow  \infty}&} 
x$ uniformly on $\Bbb{P}^n_\Bbb{C}\setminus y^\bot$, $\gamma_m ^{- 1}\xymatrix{ \ar[r]_{m \rightarrow  \infty}&} 
y$ uniformly on $\Bbb{P}^n_\Bbb{C}\setminus x^\bot$ and   there is  $\mu:y^\bot(y)\rightarrow x^\bot(x)$ such that for each
$\ell\in y^\bot(y)$ (resp. $\ell\in x^\bot(x)$ ) and $w\in \ell\setminus \{y\}$ (resp $w\in \ell\setminus \{x\}$). We have 
$\mathcal{D}_{(\gamma_m)}(w)=\mu(\ell)$ (resp.  $\mathcal{D}_{(\gamma_m^{-1})}(w)=\mu^{-1}(\ell)$). Therefore,
$\Gamma(y,x)$ is non-empty.
\end{proof}

By means of the arguments used in the proof of the previous lemma, it is not hard to show that $\Gamma(y,y)$ is a non-empty  set of 
$Isom(y,y)$.

\begin{definition}
Let $\Gamma\subset PU(1,n)$ be a  discrete group, $x,y\in \Lambda_{CG}(\Gamma)$,    $\mu \in \Gamma(y,x)$, $(\gamma_m)\subset \Gamma$ a Cartan sequence associated to $\mu$ and $T$  a projective transformation of $P(\bigwedge^2 \Bbb{C}^{n+1})$, we will say that $T$ is  a Cartan extension for $\mu$ with respect $(\gamma_m)$. If   
\begin{enumerate}
\item The Cartan decomposition of $\gamma_m$ is $k^+_m A_m k^-_m$.
\item We have $k_m^\pm \xymatrix{ \ar[r]_{m \rightarrow  \infty}&} k^\pm $, where $k^+, k^-\in K=PU(n + 1) \cap PU(1, n).$
\item Also $T=k^+H k^-$.
\end{enumerate}
Moreover, we will say that $T$ is a Cartan extension for  $\mu$ if there is a Cartan sequence $(\gamma_m)$  associated to $\mu$ such that 
$T$ is a Cartan extension for $\mu$ with respect $(\gamma_m)$.

\end{definition}

The proof of the next lemma follows straightforwardly and we omit it here.

 \begin{lemma}
Let $\Gamma\subset PU(1,n)$ be a  discrete group, $x,y\in \Lambda_{CG}(\Gamma)$ and    $\mu \in \Gamma(y,x)$. If $T$ is a Cartan extension  for $\mu$, then $T(\ell)=\mu(\ell)$ for each $\ell \in y^\bot (y)$.
\end{lemma}

\begin{lemma}
Let $\Gamma\subset PU(1,n)$ be a  discrete group and $x,y\in \Lambda_{CG}(\Gamma)$ be distinct points, then  $\Gamma(y,x)$ is a closed 
subset of $Isom(y,x)$.
\end{lemma}

\begin{proof}
 Let $(\mu_m)\subset \Gamma(y,x)$  and $\mu\in Isom(y,x)$ such that $ \mu_m \xymatrix{ \ar[r]_{m \rightarrow  \infty}&} \mu $. For each $m\in \Bbb{N}$ let $(\gamma_{jm})\subset \Gamma$ be a Cartan sequence associated to $\mu_m$ satisfying:
 
\begin{enumerate}
\item The Cartan decomposition of $\gamma_{jm}$ is given by $\gamma_{jm}=k^+_{jm}A_{jm}k^-_{jm}$.

\item There are 
$k^+_m, k^-_m \subset U(n+1)\cap U(1,n)$  
such that  
\[
d(k^\pm_{jm}, k^\pm_m)< 2^{-j}\\
\]

\end{enumerate}

Therefore,  $T_m=k^+_mH k^-_m $  is a Cartan extension for $\mu_m$ with respect  $(\gamma_{jm})$. Since $K$ is compact, there is a 
strictly increasing sequence $(n_m)\subset \Bbb{N}$ and  $k_+,k_-\in K=PU(n + 1) \cap  P U(1, n),$ such that 
\[
d(k_{n_i}^\pm, k_\pm)< 2^{-i}.
\]
Therefore, if $\ell \in y^\bot(y)$ then  
$$
\mu_{n_m}(\ell)= T_{n_m} (\ell) 
\xymatrix{ \ar[r]_{m \rightarrow  \infty}&} 
 k_+( H(k_-(\ell))). 
$$
Therefore, taking  $T=k_+ Hk_-$ we have $T(\ell)=\mu(\ell)$
 for each $\ell \in y^\bot(y)$. In order to conclude  the proof,  observe there is a strictly increasing sequence $l_m$ such that 
 $(\gamma_{l_m n_m})$ is a Cartan sequence associated to $\mu$.
\end{proof}

\begin{remark}
Given $\Gamma\subset PU(1,n)$ a  discrete subgroup and 
$y\in \Lambda_{kul}(\Gamma)$, by means of the arguments used in the proof of the previous lemma, it is not hard to show that $\Gamma(y,y)$ is a non-empty closed set of $Isom(y,y)$.
\end{remark}

Let us define the following binary operation in $Bihol(p,q)$

\begin{definition}
Let $p,q\in \partial{\Bbb{H}^n_\Bbb{C}}$, let 
\[
Bihol(p,q)=
\{
\mu:p^\bot(p)\rightarrow q^\bot (q)\vert \mu \textrm{ is biholomorphic}
\}
\]
We define the  binary operation $\star$  on $Bihol(p,q)$
\[
\mu\star \nu (\ell)
=
\left \{
\begin{array}{ll}
\mu  (\nu(\ell)) & \textrm{ if } p=q\\
\mu(Span(\{\nu(\ell)\cap p^\bot(p),p\})
\end{array}
\right .
\]
\end{definition}

The following result will show that the operation $\star$ is quite usual.

\begin{lemma}
Given $p,q\in \partial \Bbb{H}^n_\Bbb{C}$, the algebraic structure $(Bihol(p,q),\star)$ is a group isomorphic to 
$(PSL(n-1,\Bbb{C}),\circ)$. Moreover, $(Isom(p,q),\star)$ is an isomorphic group to $(PU(n-1),\circ)$.
\end{lemma}

\begin{lemma} \label{l:monoid}
Let $\Gamma\subset PU(1,n)$ be a non-elementary discrete group and $x,y\in \Lambda_{CG}(\Gamma)$ be distinct points, then 
$\Gamma(y,x)$ is closed under the operation $\star$. In particular, $\Gamma(y,x)$ is a closed monoid.
\end{lemma}

\begin{proof}
Let $\gamma_1,\gamma_2\in \Gamma(y,x)$ and  $(\gamma_{1m}), (\gamma_{2m})\subset \Gamma$  be the Cartan sequences associated to 
$\gamma_1$ and $\gamma_2$, respectively. A straightforward calculation shows $\gamma_{1m}\gamma_{2m}\xymatrix{ \ar[r]_{m \rightarrow  \infty}&} x $ uniformly on compact  sets 
of $\Bbb{P}^n_\Bbb{C}\setminus y^\bot$ and $\gamma_{2m}^{-1}\gamma_{2m}^{-1}\xymatrix{ \ar[r]_{m \rightarrow  \infty}&} y $ uniformly 
on compact  sets of $\Bbb{P}^n_\Bbb{C}\setminus x^\bot$. By Proposition \ref{p:llemaleq}, there is  a strictly  increasing sequence 
$(n_m)\subset \Bbb{N}$ and $\nu\in \Gamma(y,x)$ such that for every $\ell\in y^\bot (y)$ and $w\in \ell\setminus\{y\}$ we have
\[
\begin{array}{l}
\mathcal{D}_{(\gamma_{1n_m}\gamma_{2n_m})}(w)=\nu(\ell)\\
\end{array}
\]

Let $\ell \in y^\bot(y)$, $w\in \ell\setminus \{y\}$ and $(w_m)\subset \Bbb{P}^n_\Bbb{C}$ such that 
$w_m \xymatrix{ \ar[r]_{m \rightarrow  \infty}&} w $. Also, let us assume that there are $w_1,w_2\in \Bbb{P}^2_{\Bbb{C}}$ such that 
\[
\begin{array}{l}
\gamma_{2n_m}(w_m)\xymatrix{ \ar[r]_{m \rightarrow  \infty}&} w_1\\
\gamma_{1n_m}(\gamma_{2n_m}(w_m))\xymatrix{ \ar[r]_{m \rightarrow  \infty}&} w_2.
\end{array}
\]
Observe that in case  $w_1\notin y^\bot $, then $w_2=x$. On the other hand,  when $w_1\in y^\bot $, we have  
$w_1=\gamma_2(\ell)\cap y^\bot$. Consequently, $w_2\in \gamma_1(Span(\{\gamma_2(\ell)\cap y^\bot,y\}))$. Therefore, 
$\nu=\gamma_1\star\gamma_2$.
\end{proof}

\begin{theorem} \label{t:monoid}
Let $\Gamma\subset PU(1,n)$ be a non-elementary discrete group and $x,y\in \Lambda_{CG}(\Gamma)$ be distinct points, then  
$\Gamma(y,x)$ is a compact Lie group.
\end{theorem}

\begin{proof}
We will show $\Gamma(y,x)$ is indeed a group. Since $(Isom(p,q),\star)$ is isomorphic to $(PU(n-1),\circ)$. We will also show that
$\Gamma(y,x)$ is a subgroup of $(PU(n-1),\circ)$ after identifying $(Isom(p,q),\star)$ with $(PU(n-1),\circ)$. Please note 
$(Isom(p,q),\star)$ is also compact.

Let $a\in\Gamma(y,x)$ be an arbitrary element, and consider the set 
\begin{equation*}
	A:=\overline{\lbrace a, a^2,a^3,...,a^n,..\rbrace}.
\end{equation*}
One can easily see $A\subset\Gamma(y,x)\subset PU(n-1)$ is a closed Abelian subsemigroup. Since $PU(n-1)$ is compact, $A$ is also 
compact which implies $A$ is an Abelian abundant semigroup. That is, $A$ has minimal ideals and each ideal contains an idempotent 
element. Notice that $s=\mathrm{id}$ is the only element in $PU(n-1)$ satisfying $s^2=s$. 
Now, let $\mathcal{I}\subset A$ be a minimal ideal of $A$, and thus $\mathrm{id}\in\mathcal{I}$. It is clear that
$\mathcal{I}$ is a group because for any element $t\in\mathcal{I}$, we have 
$t\mathcal{I}=\mathcal{I}=\mathcal{I}t$, since $\mathcal{I}$ is a two-sided minimal ideal and cannot be smaller. By
Green's theorem (in context of Green's relations), one can see $\mathcal{I}$ is a subgroup of $A$. 
Since $\mathcal{I}$ is an ideal with $\mathrm{id}$, $\mathcal{I}=A$. Now, $A$ is a group and $a^{-1}\in A\subset\Gamma(y,x)$ which 
completes the proof.
\end{proof}

\begin{definition}
Let  $\Gamma\subset PU(1,n)$ be a non-elementary discrete subgroup,  a point  
$y\in \Lambda_{Kul}(\Gamma)$ and $\mu_0\in \Gamma(y,y)$.  We will say that  $\mu_0$ is  a $y$-Cartan limit if there are:
\begin{enumerate} 
\item A sequence 
   of distinct elements $(y_m)\subset \Lambda_{Kul}(\Gamma)$  converging to $y$.

\item For each $m\in \Bbb{N}$ a  Cartan map $\mu_m\in\Gamma(y_m,y)$.

 \item For each $m\in \Bbb{N}\cup \{0\}$   a Cartan extension $T_m$ for $\mu_m$. 
\end{enumerate}
Such that $T_m$ converges to $T_0$ as projective transformations. Also, we will say that $(\mu_m)\subset \bigcup_{x,y\in \Lambda_{CG}(\Gamma)}\Gamma(x,y)=\mathfrak{C}(\Gamma)$ converges to $\mu_0\in \mathfrak{C}(\Gamma)$, if for each $m\in \Bbb{N}\cup \{0\}$ there is a Cartan extension $T$ of $\mu_m$ such that $T_m \xymatrix{ \ar[r]_{m \rightarrow  \infty}&} T _0$, as projective transformations.
\end{definition}

\begin{lemma} \label{l:prin}
Let $\Gamma\subset PU(1,n)$ be a non-elementary discrete subgroup and $y\in \Lambda_{CG}(\Gamma)$, then  
 $$
\Gamma(y)=
\overline{
\left \{
\mu :y^\bot(y) \rightarrow  y^\bot(y)\mid \mu  \textrm{ is a finite composition of } y\textrm{-Cartan limits}
\right \}
}
$$ 
is a closed monoid contained in $\Gamma(y,y)$.
\end{lemma}
\begin{proof} First, let us show that $\Gamma(y)$ is non-empty. Let $(y_m)\subset \Lambda_{CG}(\Gamma)$ be a  sequence of distinct elements such that $y_m \xymatrix{ \ar[r]_{m \rightarrow  \infty}&} y$. For  each $m\in \Bbb{N}$ let  $\mu_m\in \Gamma(y_m,x)$ and   $(\gamma_{jm})\subset \Gamma$  a Cartan sequence associated to $\mu_m$ such that:
 
\begin{enumerate}
\item The Cartan decomposition of $\gamma_{jm}$ is given by $\gamma_{jm}=k^+_{jm}A_{jm}k^-_{jm}$.

\item There are 
$k^+_m, k^-_m \in K$  
such that  
\[
d(k^\pm_{jm}, k^\pm_m)< 2^{-j}\\
\]

\end{enumerate}

Thus $T_m=k^+_mH k^-_m $ is a Cartan  extension of $\mu_m$. Since $K$ is compact   there are $k_+,k_-\in U(n+1)\cap U(1,n)$  and a strictly increasing sequence $(n_m)\subset \Bbb{N}$ such that 
\[
d(k_{n_m}^\pm, k_\pm)< 2^{-n_m}.
\]
Therefore, 
$$
T_{n_m}
\xymatrix{ \ar[r]_{m \rightarrow  \infty}&} 
 k_+Hk_-=T. 
$$
Clearly, $T(y^\bot (y))=y^\bot (y)$. Now, let $(l_m)\subset \Bbb{N}$ be a strictly increasing sequence  such that 
$(\gamma_{l_m n_m})$ is a sequence tending simply to infinity.   Then is straightforward that $(\gamma_{l_m n_m})$ is a  Cartan sequence 
associated to $T\mid_{y^\bot (y)}$, which concludes this part of the proof.\\

In order to conclude the proof, it will suffice to show that the product of two $y$-Cartan limits is in $\Gamma(y,y)$. 
Let $\mu_1,\mu_2\in Isom(y,y)$ be  $y$-Cartan limit maps, then for each $i\in \{1,2\}$ there are sequences $(y_{mi})\subset \Lambda_{CG}(\Gamma)$ and  $(\mu_{mi})\in \mathfrak{C}(\Gamma)$ such that $(y_{mi})$ is a sequence of distinct elements converging to $y$,  $\mu_{mi}\in \Gamma(y_{mi},y)$ and $\mu_{mi}\xymatrix{ \ar[r]_{m \rightarrow
\infty}&}  \mu_i $. For each $m\in \Bbb{N}$ and $i\in \{1,2\}$ let $(\gamma_{jmi})_{j\in \Bbb{N}}$ be a Cartan sequence associated to
$\mu_{mi}$ satisfying:
\begin{enumerate}
\item The Cartan decomposition of $\gamma_{jmi}$ is $k^+_{jmi}A_{jmi}k^-_{jmi}$.
\item There are $k^+_{mi},k^-_{mi}\in K$ such that
\[
d(k^\pm_{jmi}, k^\pm_{mi})<2^{-j}.
\]
\item There are $k^+_{i},k^-_{i}\in K$ such that
\[
d(k^\pm_{mi}, k^\pm_{i})<2^{-m}.
\]
\item  The projective transformation, $T_{mi}=k^+_{mi}Hk^-_{mi}$, is a Cartan extension of $\mu_{mi}$.
\item  The projective transformation, $T_i=k^+_{i}Hk^-_{i}$, is a Cartan extension of $\mu_{i}$.
\end{enumerate}
For each $m\in \Bbb{N}$, let us consider the sequence $(\tau_{jm}=\gamma_{jm1}\gamma_{jm2})_{j\in\Bbb{N}}$, then it is not hard to show
$\tau_{jm}\xymatrix{ \ar[r]_{j \rightarrow
\infty}&}  y$ uniformly on $\Bbb{P}^n\setminus y_{m2}^\bot$ and $\tau_{jm}^{-1}\xymatrix{ \ar[r]_{j \rightarrow
\infty}&}  y_{m2}$ uniformly on $\Bbb{P}^n\setminus y^\bot$. Let $n_m\subset \Bbb{N} $ be a strictly increasing sequence  such that $(\tau_{n_m m})$ is a sequence tending simply to infinite, then $(\gamma_{n_m mi})$ is a Cartan sequence associated to $\mu_i$ and 

\[
\begin{array}{l}
\tau_{n_m m}\xymatrix{ \ar[r]_{m \rightarrow \infty}&}  y \textrm{ uniformly on } \Bbb{P}^n\setminus y^\bot\\
\tau_{n_m m}^{-1}\xymatrix{ \ar[r]_{m \rightarrow \infty}&}  y \textrm{ uniformly on } \Bbb{P}^n\setminus y^\bot.
\end{array}
\]

 By Proposition \ref{p:llemaleq}, there is   $\nu\in \Gamma(y,y)$ such that for every $\ell\in y^\bot (y)$ and $w\in \ell\setminus\{y\}$ we have
\[
\begin{array}{l}
\mathcal{D}_{(\gamma_{n_m m1}\gamma_{n_m m2})}(w)=\nu(\ell)\\
\mathcal{D}_{(\gamma_{n_m m 2}^{-1}\gamma_{n_m m1}^{-1})}(w)=\nu^{-1}(\ell).
\end{array}
\]

Let $\ell \in y^\bot(y)$, $w\in \ell\setminus \{y\}$ and $(w_m)\subset \Bbb{P}^n_\Bbb{C}$ such that $w_m \xymatrix{ \ar[r]_{m \rightarrow  \infty}&} w $, also let us assume that there are $w_1,w_2\in \Bbb{P}^2_{\Bbb{C}}$ such that 
\[
\begin{array}{l}
\gamma_{n_m m2}(w_m)\xymatrix{ \ar[r]_{m \rightarrow  \infty}&} w_1\\
\gamma_{n_m m1}(\gamma_{n_m m2}(w_m))\xymatrix{ \ar[r]_{m \rightarrow  \infty}&} w_2.
\end{array}
\]
Observe that in case  $w_1\notin y^\bot $, then $w_2=y$. On the other hand,  when $w_1\in y^\bot $, we have  $w_1\in \mu_2(\ell)$. 
Consequently, $w_2\in \mu_1(\mu_2(\ell))$, therefore $\nu=\mu_1\circ\mu_2$.
\end{proof}

We have the following straightforward results.

\begin{corollary}\label{r:carcom} 
The space $ \mathfrak{C}(\Gamma)$ is sequentially compact.
\end{corollary}

\begin{corollary}\label{r:carcom} 
Let $(y_m)\subset \Lambda_{CG}(\Gamma) $ be a sequence of distinct elements converging to $y$, the there is a subsequence $(z_m)\subset (y_m)$, such that the identity $I_m\in \Gamma(z_m)$ converges to the identity $I\in\Gamma(y,y). $
\end{corollary}

Before we end this section, let us show that indeed $\Gamma(y)$ is a group.

\begin{theorem}\label{t:yyy}
Let $\Gamma\subset PU(1,n)$ be a non-elementary discrete subgroup and $y\in \Lambda_{CG}(\Gamma)$, then  
$\Gamma(y)=\Gamma(y,y)$. In particular $\Gamma(y)$ is a compact Lie subgroup naturally embedded in $PU(n-1)$.
\end{theorem}

\begin{proof} 
In order to conclude the proof, it will suffice to show that every $\mu_{-1}\in  \Gamma(y,y)$ is indeed  a  $y$-Cartan limit. 
Let $(y_{m})\subset \Lambda_{CG}(\Gamma)$  be a sequence of distinct elements converging to $y$, such that if 
$\mu_{m}\in \Gamma(y_{m},y_m)$ is  the identity of  $(Bihol(y_m,y_m),\star)$. Then   $\mu_m\xymatrix{ \ar[r]_{m \rightarrow \infty}&}  	\mu_0$, where $\mu_0$ is the identity of $(Bihol(y,y),\star)$. For each $m\in \Bbb{N}$ let us define
\[
\begin{array}{l}
\mu_{-1}\star \mu_m:y_m^\bot(y_m)\rightarrow y^\bot(y)\\
\mu_{-1}(Span(\{\mu_m(\ell)\cap y^\bot(y),y\}),
\end{array}
\]
 as in the proof of Lemma \ref{l:monoid}, one can easily show that  $\mu\star\mu_m\in \Gamma(y_m,y)$. For each $m\in \Bbb{N}$, let $S_m$ be a Cartan extension of 
$\nu_m=\mu\star\mu_m$. By Corollary \ref{r:carcom}, we can assume that there is $S_0\in \mathfrak{C}(\Gamma)$ such that 
$S_m\xymatrix{ \ar[r]_{m \rightarrow \infty}&}  S_0$ as projective transformations. Moreover, we can say that there is 
$\nu_0\in \Gamma(y)$ such that $S_0$ is the Cartan extension of $\nu_0$. Consequently,
$\nu_m \xymatrix{ \ar[r]_{m \rightarrow \infty}&} \nu_0$.  For each element $m\in \Bbb{N}\cup\{0,-1\}$,  let 
$(\gamma_{jm})_{j\in \Bbb{N}}$ be a Cartan sequence associated to $\mu_{m}$ satisfying:
\begin{enumerate}
\item The Cartan decompositions of $\gamma_{jm}$ and $\gamma_{j,-1}\gamma_{jm}$, respectively, are $k^+_{jm}A_{jm}k^-_{jm}$
and  $r^+_{jm}B_{jm}r^-_{jm}$.

\item There are $k^+_{m},k^-_{m},r^+_{m},r^-_{m}\in K$ such that
\[
max\{
d(k^\pm_{jm}, k^\pm_{m}),
d(r^\pm_{jm}, r^\pm_{m})\}<2^{-j}.
\]
\item There are $k^+_0,k^-_0,r^+_0,r^-_0\in K$ such that
\[
max\{d(k^\pm_{m}, k^\pm),
d(r^\pm_{m}, r^\pm)\}<2^{-m}.
\]
\item  We have  $T_{m}=k^+_{m}Hk^-_{m}$ and $S_m=r^+_{m}Hr^-_{m}$.
\end{enumerate}

Let us consider $\sigma_m=\gamma_{m,-1}\gamma_{mm}$, then it is not hard to show that 
$\sigma_m \xymatrix{ \ar[r]_{m \rightarrow \infty}&} y$, uniformly on compact sets of 
$\Bbb{P}^n_\Bbb{C}\setminus y^\bot.$ Let $\ell \in y^\bot(y)$, $w\in \ell\setminus \{y\}$ and $(w_m)\subset \Bbb{P}^n_\Bbb{C}$ such
that $w_m \xymatrix{ \ar[r]_{m \rightarrow  \infty}&} w $. Also, let us assume that there are $w_1,w_2\in \Bbb{P}^2_{\Bbb{C}}$ such 
that 
\[
\begin{array}{l}
\gamma_{mm}(w_m)\xymatrix{ \ar[r]_{m \rightarrow  \infty}&} w_1\\
\gamma_{m,-1}(\gamma_{mm}(w_m))\xymatrix{ \ar[r]_{m \rightarrow  \infty}&} w_2.
\end{array}
\]
In case  $w_1\in \ell$, then  $w_2\in \mu_{-1}(\ell) $. Therefore, $\nu_0=\mu_{-1}\star\mu_0$.
\end{proof}

\section{Proof of the main theorem}\label{s:main}
Le us consider the following result:

\begin{theorem} \label{t:pne}
Let $\Gamma\subset PU(1,n)$ be a non-elementary discrete subgroup, then
\begin{enumerate}
\item The Kulkarni discontinuity region $\Omega_{Kul}(\Gamma)$ is the largest open set on  which $\Gamma$ acts properly and 
discontinuously and agrees with $Eq(\Gamma)$.

\item The Kulkarni limit set  of $\Gamma$ can be described as the union of hyperplanes tangent to $\partial \Bbb{H}^n_\Bbb{C}$ at  points in the 
Chen-Greenberg limit set of $\Gamma$, {\it i.e.},

\[
\Lambda_{Kul}(\Gamma)=\bigcup_{p\in \Lambda_{CG}(\Gamma)}p^\bot.
\]
\end{enumerate}
\end{theorem}
\begin{proof} 
On the contrary, let  us assume that $Eq(\Gamma)\neq \Omega_{Kul}(\Gamma)$. By Proposition \ref{p:pkg} and Corollary \ref{c:eqpu}, we 
conclude that there is  $p\in \Lambda_{CG}(\Gamma)$ such that 
$p^\bot \not\subset \Lambda_{Kul}(\Gamma)$. Then there is $w_0\in p^\bot\setminus \{p\}$ such that $w_0\in \Omega_{Kul}(\Gamma)$. 
From Lemma \ref{l:prin}, we have  $Id_{p^\bot(p)}\in \Gamma(p,p)$, then  there is a  Cartan sequence  $(\gamma_m)\subset \Gamma$ 
associated to $Id_{p^\bot(p)}$. Consequently, $\overleftrightarrow{w_0,p}=\mathcal{D}_{(\gamma_m)}(w_0)\subset \Lambda_{kul}(\Gamma)$, 
which is a contradiction. Therefore, $\Omega_{Kul}(\Gamma)=Eq(\Gamma)$ and   
\[
\Lambda_{Kul}(\Gamma)=\bigcup_{p\in \Lambda_{CG}(\Gamma)}p^\bot.
\]
Through similar arguments, we can show that $\Omega_{Kul}(\Gamma)$ is the largest open set on which $\Gamma$ acts properly 
discontinuously.
\end{proof}

\subsection*{Proof of Theorem \ref{t:principal}}
In virtue of Theorem \ref{t:pne}, we only need to show the theorem when  $\Gamma$ is elementary. Consider the following cases:\\

Case 1.- The Chen-Greenberg limit set of $\Gamma$ is a single point. Let   $p\in \Lambda_{CG}(\Gamma)$  be the unique point, by Proposition   \ref{p:llemaleq}  and  arguments as  in  Theorem \ref{t:monoid},   we can ensure that  $Id_{p^\bot(p)}\in \Gamma(p,p)$. If there is  $w_0\in p^\bot \setminus \Lambda(\Gamma)$, let   $(\gamma_m)\subset \Gamma$ be a Cartan sequence  
associated to $Id_{p^\bot(p)}$. Thus  $\overleftrightarrow{w_0,p}=\mathcal{D}_{(\gamma_m)}(w_0)\subset p^\bot$, which proves the theorem in this case.\\

Case 2.- The Chen-Greenberg limit set of $\Gamma$ has exactly two points.
After conjugating with an element in $PU(1,n)$, if it is necessary, we can assume that $\{[e_1],[e_{n+1}]\}=\Lambda_{CG}(\Gamma)$. Let $\Gamma_0=Isot(\Gamma,[e_1])\cap Isot(\Gamma,[e_{n+1}])$, thus $\Gamma_0$ is a subgroup of $\Gamma$ with finite index, therefore $\Lambda_{Kul}(\Gamma_0)=\Lambda_{Kul}(\Gamma)$. Let $(\gamma_m)_{m\in \Bbb{N}}\subset \Gamma$ be sequence of distinct elements, then 
\begin{equation} \label{e:lox2}
\gamma_m=
\begin{bmatrix}
r_m c_m\\
& u_m\\
&&r_m^{-1} c_m
\end{bmatrix}
\end{equation}
where $c_m^2=det (U_m)$, $r_m\in \Bbb{R}^+$ and $U_{m}\in U(n-1)$. From Equation  \ref{e:lox2} we conclude $L_0(\Gamma_0)=L_1(\Gamma_0)$ and $L_2(\Gamma)=[e_1]^\bot \cup [e_{n+1}]^\bot$, which concludes the proof.
$\square$\\

As a corollary of the main theorem, we have the following result, compare with results in \cite{BN,CS}.

\begin{corollary}
Let $\Gamma\subset PU(1,n)$  be a discrete group such that $\Gamma$ is irreducible, then each connected component of 
$\Omega_{Kul}(\Gamma)$ is  a complete  Kobayashi hyperbolic space, pseudoconvex, a domain of holomorphy and a Stein Manifold. Moreover,
$\Lambda_{CG}(\Gamma)\subset \Lambda_{Kul}(\Gamma)$ is the unique minimal closed set for the action of $\Gamma$ on $\Bbb{P}^n_\Bbb{C}$. 
\end{corollary}

\section{The Cartan invariant and the control group } \label{s:cartan}
 In this section, we will show  how the control group ``encodes the geometry" of the Chen-Greenberg limit set trough the Cartan 
 angular invariant.
\begin{definition}
Given   $z=[z_1,z_2,\ldots,z_{n+1}]\in \partial \Bbb{H}^n_{\Bbb{C}}\setminus\{e_1,e_{n+1}\}$
we define:
\[
\begin{array}{c}
\phi_z:Span(\{e_2,\ldots, e_{n}\}) \rightarrow Span(\{e_2,\ldots, e_{n}\})\\
\phi_z(w)= 
Span
((
Span((
Span(\{w,e_{n+1}\})
\cap z^\bot)
\cup \{z\})\cap e_1^\bot)\cup \{ e_1\}
)
\cap 
e_{n+1}^\bot
\end{array}
\]
\end{definition}
 The following lemma is the computation of the transformation  $\phi_z$.

\begin{lemma} \label{r:mat}
Let $z=[z_1,z_2,\ldots,z_{n+1}]\in \partial\Bbb{H}^n_{\Bbb{C}}\setminus\{e_1,e_{n+1}\}$, then $\phi_z(w)=[[M_z]](w)$, where 

 \[
M_z=
 \begin{pmatrix}
z_{n+1}\bar{ z}_1+\vert z_2\vert^2& z_2\bar{z}_3&...& z_2\bar{z}_n\\
 z_3\bar{z}_2&z_{n+1}\bar{ z}_1+\vert z_3\vert^2&...& z_3\bar{z}_n\\
 \vdots&\vdots&\ddots&\vdots\\
 z_n\bar{z}_2& z_n\bar{z}_3&\ldots& z_{n+1}\bar{ z}_1+\vert z_n\vert^2
 \end{pmatrix}
 \]
\end{lemma}
\begin{proof}
Let $w\in Span\{e_2,\ldots, e_n\}$, then we can assume $w=[0,w_2, \ldots,w_n,0]$, thus
 \[
 Span(\{w,e_{n+1}\})=\{[0, (1-\lambda)w_2,\ldots,(1-\lambda)w_n,\lambda]\vert  \lambda\in\hat{\Bbb{C}}\}.
 \]
To obtain $w_1=Span(\{w,e_{n+1}\})\cap z^\bot$, we use the Hermitian form 
$\langle,\rangle$,  so $w_1$ is obtained through  the solution of 

$$
\lambda \bar{z}_1+
\sum_{j=2}^n(1-\lambda)w_j\bar{z}_j=0.
$$

 And the solution is 
 \begin{equation}
 \sigma=
\frac
{
\sum_{j=2}^nw_j\bar{z}_j
}
{
-\bar{z}_1+\sum_{j=2}^nw_j\bar{z}_j
}.
 \end{equation}
 Therefore, \[w_1=[0, (1-\sigma)w_2, \ldots,(1-\sigma)w_n,\sigma].\]
 If $x \in Span(\{w_1,z\})$, we have 
\[x=
[(1-\lambda)z_1, \lambda(1-\sigma)w_2+(1-\lambda)z_2, \ldots,
 \lambda(1-\sigma)w_n+(1-\lambda)z_n, \lambda\sigma+ (1-\lambda)z_{n+1}],
\] where  $\lambda\in \hat{\Bbb{C}}$. Thus, $w_2= Span(\{w_1,z\})\cap e_1^\bot$ is given by the solution of 
 \begin{equation}
 0=\lambda\sigma+(1-\lambda)z_{n+1}.
 \end{equation}
Such  solution is $\eta=\frac{z_{n+1}}{z_{n+1}-\sigma}$. Now, 
\[w_2=[(1-\eta)z_1, \eta(1-\sigma)w_2+(1-\eta)z_2, \ldots,
 \eta(1-\sigma)w_n+(1-\eta)z_n, \eta\sigma+ (1-\eta)z_{n+1}].\]

Next,  if  $y\in Span(\{w_2,e_1 \})$, then 

 \[\begin{split}
y=&[\lambda(1-\eta)z_1+(1-\lambda), \lambda(\eta(1-\sigma)w_2+(1-\eta)z_2), \\&\ldots,
\lambda(\eta(1-\sigma)w_n+(1-\eta)z_n), \lambda(\eta\sigma+ (1-\eta)z_{n+1})]
 \end{split}
 \]
and $\phi_z(w)=Span(\{w_2,e_1 \})\cap e_{n+1}^\bot$ is induced by the solution of 
 \[
\lambda(1-\eta)z_1+(1-\lambda) =0,
\]
which is $\lambda=\frac{1}{\eta z_1}$.

 A straightforward calculation shows that
 \[
 \phi_z(w)
= \left [ \sum_{i=2}^n( (z_{n+1}\bar{ z}_1+\vert z_i\vert^2 ) w_i+z_i\sum_{j\neq 1,i}w_j\bar{z}_j)e_i \right ],
 \]
which concludes the proof.
 \end{proof}
 We have the following straightforward corollary.

\begin{corollary}
 If $z\in \Bbb{H}^n_\Bbb{C}\setminus \{e_1,e_{n+1}\}$. Then,
\begin{enumerate}

\item We have  $\phi_z=Id$ if and only if $z\in Span(\{e_1,e_{n+1}\})$

\item Also, $Det(M_z-z_{n+1}\bar{z}_1 Id)=0$.

\end{enumerate}
\end{corollary}

Now, we get the following properties of $M_z$.

\begin{lemma} \label{l:cartanang}
Let $z\in \Bbb{H}^n_\Bbb{C}\setminus \{e_1,e_{n+1}\}$. Then,
\begin{enumerate}
\item \label{i:car1} We have $[[M_z]]\in PU(n-1)$.
\item \label{i:car2} The projective subspace $W_z=e_1^\bot\cap e_{n+1}^\bot\cap z^\bot$, satisfies
$W_z\subset Fix [[M_z]]$. 
\end{enumerate}
Moreover, if $z=[z_1,\ldots, z_{n+1}]\notin  Span (\{e_1,e_{n+1}\})$ and  $z_0=[0,z_2,\ldots,z_{n},0]$, then 
\begin{enumerate}
\item[(a)] The vector $\widetilde{z}_0=(z_2,\ldots, z_n)$ is an eigenvector  $M_z$ with eigenvalue  $\lambda_z=z_{n+1}\bar{z}_1+\sum_{j=2}^n \vert z\vert^2=-z_{1}\bar{z}_{n+1}=\vert z_{1}\bar{z}_{n+1}\vert e^{i\Bbb{A}(e_1,z,e_{n+1})}$.

\item[(b)]   We have  $Fix(\phi_z)=W_z\cup \{z_0\}$.

\end{enumerate}

\end{lemma}
\begin{proof} Let us show part (\ref{i:car1}).
Taking $z=[z_1,z_2,\ldots,z_{n+1}]$, by Lemma \ref{r:mat} and a straightforward calculation, we have
\[
M_z M_z^*
=
 \begin{pmatrix}
\vert z_{n+1}\vert^2 \vert  z_1\vert^2& 0&\ldots & 0\\
0&\vert z_{n+1}\vert^2 \vert  z_1\vert^2&\ldots & 0\\
 \vdots&\vdots&\ddots&\vdots\\
0&0&\ldots&  \vert z_{n+1}\vert^2 \vert  z_1\vert^2
 \end{pmatrix}
 \]
which shows the  claim.\\

Let us show part (\ref{i:car2}). Let $w\in W_z$, then 
\[
\begin{array}{l}
Span(\{w,e_{n+1}\})\cap z^\bot=\{w\}\\    
Span(\{w,z\})\cap e_1^\bot=\{w\}\\
Span(\{w, e_1\})\cap e_{n+1}^\bot=\{w\}
\end{array}
\]
which shows $\phi_z(w)=w$, as required.\\

The proof of part (a) is a straightforward calculation, so we omit it here.  To conclude the proof, let us show part (b). 
Since $\phi_z$ is not the identity, we conclude that  $W_z$ is a projective space of dimension $n-3$ contained in 
$Span\{e_2,\ldots, e_n\}$. On the other hand, since $\phi_z\in PU(n-1)$, it suffices to 
show $z_0\notin W_z$. If this is not true, we must have $z_0\in z^\bot$, which is equivalent to the fact 
$\sum_{j=2}^n\vert z_j\vert^2=0$, which is not possible. Therefore, $z_0\notin W_z$.
\end{proof}

As an easy corollary, we get
\begin{corollary}
Let $z\in \Bbb{H}^n_\Bbb{C}\setminus \{e_1,e_{n+1}\}$, then there is $H:e_1^\bot \cap e_{n+1}^\bot\rightarrow e_1^\bot \cap e_{n+1}^\bot $  a projective transformation such that:
 \[
M\phi_z M^{-1}=
\left [
\begin{array}{llll}
-e^{i\Bbb{A}(z,e_1,e_{n+1})}&\\
                           & \ddots \\
                           &            & -e^{i\Bbb{A}(z,e_1,e_{n+1})} \\
                           &            &                                    &e^{i\Bbb{A}(e_1,z,e_{n+1})}
\end{array}
\right ].
\]
\end{corollary}

\begin{theorem}
Let $\Gamma\subset PU(1,n)$ be a non-elementary discrete subgroup and $x,y,z\in \Lambda_{CG}(\Gamma)$ be three distinct elements, set
\[
\begin{array}{c}
\phi_{xyz}:x^\bot\cap z^\bot \rightarrow x^\bot\cap z^\bot\\
\phi_{xyz}(w)= 
Span
((
Span((
Span(\{w,z\})
\cap y^\bot)
\cup \{y\})\cap x^\bot)\cup \{ x\}
)
\cap 
z^\bot,
\end{array}
\]
and
\[
\begin{array}{l}
\Phi_{xyz}:z^\bot(z)\rightarrow z^\bot(z)\\
\Phi_{xyz}(\ell)=Span(\phi_{xyz}(\ell\cap x^\bot)\cup\{z\}).
\end{array}
\]
Then
\begin{enumerate}
\item \label{i:cartan1} There is an ordered base $\beta=\{v_1,\ldots, v_{n-1}\}$ such that
$$
\phi_{xyz}\left(\left[\sum_{j=1}^{n-1}\alpha_j v_j\right ]\right)=\left [e^{i\Bbb{A}(x,y,z)}\alpha_{n-1} v_j-\sum_{j=1}^{n-2}e^{i\Bbb{A}(y,x,z)}\alpha_j v_j\right ]
$$

\item \label{i:cartan2} The transformation 
$\Phi_{xyz}$ belongs to $\Gamma(z,z)$.
\end{enumerate}
\end{theorem}
\begin{proof}
The proof of part (\ref{i:cartan1}) is straightforward, so we will omit it here.  Let us show part (\ref{i:cartan1}). Let 
$I_1\in \Gamma(z,y)$ be the identity of $(Bihol(y,z),\star)$ and $Id_2\in \Gamma(x,z)$ be the identity of $(Bihol(x,z),\star)$. Also,
let $(\gamma_{jm})$ be a Cartan  sequence associated to $I_j$. It is not hard to show that $\gamma_{2m}\gamma_{1m} \xymatrix{ \ar[r]_{m \rightarrow
\infty}&}  z $ uniformly on compact sets of $\Bbb{P}^n_\Bbb{C}\setminus z^\bot$. By Proposition \ref{p:llemaleq}, there is  a strictly
increasing sequence $(n_m)\subset \Bbb{N}$ and $\nu\in \Gamma(z,z)$ such that for every $\ell\in z^\bot (z)$ and 
$w\in \ell\setminus\{z\}$ we have
\[
\begin{array}{l}
\mathcal{D}_{(\gamma_{2n_m}\gamma_{1n_m})}(w)=\nu(\ell).\\
\end{array}
\]

Let $\ell \in z^\bot(z)$, $w\in \ell\setminus \{z\}$ and $(w_m)\subset \Bbb{P}^n_\Bbb{C}$ such that 
$w_m \xymatrix{ \ar[r]_{m \rightarrow  \infty}&} w $. Also, let us assume that there are $w_1,w_2,w_3\in \Bbb{P}^2_{\Bbb{C}}$ such 
that 
\[
\begin{array}{l}
\gamma_{1n_m}(w_m)\xymatrix{ \ar[r]_{m \rightarrow  \infty}&} w_1\\
\gamma_{2n_m}(\gamma_{1n_m}(w_m))\xymatrix{ \ar[r]_{m \rightarrow  \infty}&} w_2\\
\end{array}
\]

As in the proof of Lemma \ref{l:monoid}, we deduce 
$$
w_1\in I_1(\ell)=Span((\ell\cap y^\bot)\cup\{y\}).$$ 
Consequently, 
$$
w_2\in I_2(Span((I_1(\ell)\cap x^\bot)\cup \{x\}))=Span((Span((I_1(\ell)\cap x^\bot)\cup\{x\})\cap z^\bot)\cup \{z\})
=\Phi_{xyz}(\ell).
$$ 
Therefore, $\nu=\Phi_{xy,z}$, which concludes the proof. 
\end{proof}

We get the straightforward corollary.

\begin{corollary}
Let $\Gamma\subset PU(1,n)$ be a non-elementary discrete subgroup, set
$\Gamma(x,y,z)=\{\Phi_{x,y,z}\vert x,y,z\in \partial \Bbb{H}^n_\Bbb{C} \textrm{ are  distinct points}\}$
then
\begin{enumerate}
\item  We have $\Gamma(x,y,z)=\{Id_{z^\bot(z)}\vert z\in \Lambda_{CG}(\Gamma)\}$ if and only if $\Gamma$ is a $\Bbb{C}$-Fuchsian group.
\item  Every element in  $\Gamma(x,y,z)$ has order 2 if and only if $\Gamma$ is a $\Bbb{R}$-Fuchsian group.

\end{enumerate}
\end{corollary}

\begin{definition}
Let $y\in \Lambda_{CG}(\Gamma)$, we define the Cartan group at $y$ as 
$$
\mathfrak{C}(y)=\overline
{
\{
\Phi_{x_1y_1y}\cdots\Phi_{x_ny_ny}\vert x_i,y_i\in \Lambda_{CG}(\Gamma) \textrm{ and } x_i\neq  y_i 
\}
}
$$
\end{definition}

We will illustrate the Cartan group in a forthcoming article.
\section*{Acknowledgments}
The authors would like to thank to W. Barrera, N. Gusevskii, J. P. Navarrete, J. Parker, M. Ucan  \&  J. Seade  for 
fruitful conversations. We are gratefull to M. Mendez and  A. Guillot for introducing us into the work of C. Frances. The second author, also thanks Carlos Cabrera and J. Seade for their warm invitation to the  UCIM  at the UNAM
and to the staff of the UCIM for their kindness and help during his stay.

\bibliographystyle{amsalpha} 
\bibliography{mybib}

\providecommand{\bysame}{\leavevmode\hbox to3em{\hrulefill}\thinspace}
\providecommand{\MR}{\relax\ifhmode\unskip\space\fi MR }
\providecommand{\MRhref}[2]{%
  \href{http://www.ams.org/mathscinet-getitem?mr=#1}{#2}
}
\providecommand{\href}[2]{#2}
\begin{thebibliography}{ACCM15}

\bibitem[ACCM15]{ACCM}
V.~Alderete, C.~Cabrera, A.~Cano, and M.~Mendez, \emph{Extending the action of
  schottky groups in the complex anti-de sitter space to the projective space},
  to appear in Trends in mathematics, 2015.

\bibitem[BN09]{BN}
W.~Barrera and J.~P. Navarrete, \emph{Discrete subgroups of pu(2,1) acting on
  p2c and the kobayashi metric}, Bull. Braz. Math. Soc. (N.S.) \textbf{40}
  (2009), no.~1, 99--106.

\bibitem[CG74]{CG}
S.~S. Chen and L.~Greenberg, \emph{Hyperbolic spaces}, Contributions to
  analysis (a collection of papers dedicated to Lipman Bers), Academic Press,
  New York, 1974, pp.~49--87.

\bibitem[CNS13]{CNS}
A.~Cano, J.~P. Navarrete, and J.~Seade, \emph{Complex kleinian groups},
  Progress in Mathematics, no. 303, Birkh{\"a}user/Springer Basel AG, Basel,
  2013.

\bibitem[CS10]{CS}
A.~Cano and J.~Seade, \emph{On the equicontinuity region of discrete subgroups
  of pu(1,n)}, J. Geom. Anal. \textbf{20} (2010), no.~2, 291--305.

\bibitem[Fra05]{frances}
C.~Frances, \emph{Lorentzian kleinian groups}, Comment. Math. Helv. \textbf{80}
  (2005), no.~4, 883--910.

\bibitem[Gol99]{goldman}
W.~M. Goldman, \emph{Complex hyperbolic geometry}, Oxford Mathematical
  Monographs, Oxford University Press, New York, 1999.

\bibitem[Kam91]{kamiya}
S.~Kamiya, \emph{Notes on elements of u(1,n;c)}, Hiroshima Math. J. \textbf{21
  (1)} (1991), 23--45.

\bibitem[Kna02]{knapp}
A.~W. Knapp, \emph{Lie groups beyond an introduction}, second edition ed.,
  Progress in Mathematics, no. 140, Birkh{\"a}user Boston, Boston, MA, 2002.

\bibitem[Kul78]{kulkarni}
R.~S. Kulkarni, \emph{Groups with domains of discontinuity}, Math. Ann.
  \textbf{237} (1978), no.~3, 253--272.

\bibitem[Men15]{mendez}
M.~Mendez, Ph.D. thesis, UNAM, 2015.

\bibitem[Nav06]{jp}
J.~P. Navarrete, \emph{On the limit set of discrete subgroups of pu(2,1)},
  Geom. Dedicata \textbf{122} (2006), 1--13.

\bibitem[Ner11]{neretin}
Y.~A. Neretin, \emph{Lectures of gaussian integral operators and classical
  groups}, EMS Series of Lectures in Mathematics., European Mathematical
  Society, Z{\"u}rich, 2011.

\bibitem[SV02]{SV1}
J.~Seade and A.~Verjovsky, \emph{Higher dimensional complex kleinian groups},
  Math. Ann. \textbf{322} (2002), no.~2, 279--300.

\bibitem[SV03]{SV2}
\bysame, \emph{Complex schottky groups}, vol. 287, Asterisque, no.~xx, SMF,
  2003.

\bibitem[Uca]{ucan}
A.~Ucan, Ph.D. thesis, UNAM, Thesis in preparation.

\end{thebibliography}

\end{document}